\documentclass[12pt]{amsart}
\usepackage{mathrsfs}
\usepackage{xcolor}
\usepackage{color}
\usepackage{ulem}
\usepackage[a4paper,asymmetric]{geometry}
\usepackage{mathscinet}
\usepackage{fullpage}
\usepackage{latexsym}
\usepackage{amsthm}
\usepackage{amssymb}
\usepackage{amsfonts}
\usepackage{amsmath}
\usepackage{verbatim}
\usepackage{a4wide}
\usepackage[colorlinks,linkcolor={blue},
citecolor={blue},urlcolor={red},]{hyperref}
\usepackage{hyperref}
\newtheorem{theorem}{Theorem}[section]
\newtheorem{lemma}[theorem]{Lemma}
\newtheorem{proposition}[theorem]{Proposition}

\theoremstyle{definition}
\newtheorem{definition}[theorem]{Definition}

\theoremstyle{remark}
\newtheorem{remark}[theorem]{Remark}
\newtheorem{example}[theorem]{Example}
\numberwithin{equation}{section}
\def\d{{\rm d}}
\def\e{{\rm e}}
\def\E{\mathbb E}
\def\R{\mathbb R}
\def\P{\mathbb P}
\def\bL{\mathbb L^2}

\def\la{\left(}

\def\ra{\right)}
\def\mc{\mathscr}

\def\be{\begin{equation}\label}
\def\ee{\end{equation}}
\def\bd{\begin{definition}\label}
\def\ed{\end{definition}}
\def\bt{\begin{theorem}\label}
\def\et{\end{theorem}}
\def\bl{\begin{lemma}\label}
\def\el{\end{lemma}}

\begin{document}
\title[]{On Linear Stochastic Flows}
\author{Beniamin Goldys}
\address{School of Mathematics and Statistics, The University of Sydney, Sydney 2006, Australia}
\email{Beniamin.Goldys@sydney.edu.au}
\author{Szymon Peszat}
\address{Institute of Mathematics, Jagiellonian University, {\L}ojasiewicza 6, 30-348 Krak\'ow, Poland}
\email{napeszat@cyf-kr.edu.pl}
\thanks{This  work was partially supported by the ARC Discovery Grant DP200101866.}
\thanks{The work of Szymon Peszat    was supported by Polish National Science Center grant  2017/25/B/ST1/02584.}
\begin{abstract}
We study the existence of the stochastic flow associated to a linear stochastic evolution equation  
$$
\d X= AX\d t +\sum_{k} B_k X\d W_k,
$$
on a Hilbert space. Our first result covers the case  where $A$ is the generator of a $C_0$-semigroup, and $(B_k)$ is a  sequence of  bounded linear operators such that $\sum_k\|B_k\|<+\infty$. We also provide sufficient conditions for the existence of stochastic flows in the Schatten classes beyond the space of Hilbert-Schmidt operators. Some new results and examples concerning the so-called commutative case are presented as well. 
\end{abstract}
\maketitle
\tableofcontents


\section{Introduction}
Consider the following linear evolution equation 
\begin{equation}\label{E11}
\d X= AX\d t + \sum_{k} B_k X\d W_k,\qquad t\ge s,\,\,X(s)=x,
\end{equation}
where $(W_k)$ is a  sequence of independent standard real valued Wiener processes defined on a probability space $(\Omega, \mathfrak{F},\mathbb{P})$,  $A$ is the generator of a $C_0$-semigroup $\left(\e^{tA}\right)$  on a Hilbert space $(H,\langle \cdot,\cdot\rangle_H)$ and  $B_k$, $k=1, \ldots $,  are possibly unbounded linear operators on $H$.  By the solution to $\eqref{E11}$ we understand the so-called mild solution defined as  an  adapted to the filtration generated by $(W_k)$ and having continuous trajectories in $H$ process $X^x_s(t)$, $t\ge s$, satisfying the integral equation  
$$
X^x_s(t)= \e^{(t-s)A}x +\sum_{k} \int_s^t \e^{(t-r)A} B_k X^x_s(r)\d W_k(r), \qquad t\ge s.  
$$
For more details see the books of Da Prato and Zabczyk \cite{DaPrato-Zabczyk} and   Flandoli \cite{Flandoli}. 

Let us denote by $L(H,H)$ the space of all bounded linear operators on $H$. We denote by $\|\cdot \|$ the operator norm on $L(H,H)$ and by $\vert \cdot \vert _H$ the norm on $H$.  Let $\Delta := \{(s,t)\colon 0\le s\le t<+\infty\}$. 
\begin{definition}
We say that \eqref{E11} defines a \emph{stochastic flow} if there exists a  mapping  $\mathcal{X}\colon\Delta\times \Omega\mapsto L(H,H)$  such that: 
\begin{itemize}
\item[(i)]  for every $s\ge 0$ and $x\in H$, the process $\mathcal{X}(s,t;\cdot)(x)$, $t\ge s$, has continuous trajectories in $H$,  $\P$-a.s.,
\item[(ii)]  for every $s\ge 0$ and $x\in H$, we have $\mathcal{X}(s,t)(x)= X^x_s(t)$, for all $t\ge s$,   $\mathbb{P}$-a.s., 
\item[(iii)] for all $0\le s\le t\le r$ and $\omega\in \Omega$, $\mathcal{X}(t,r;\omega)\circ  \mathcal{X}(s,t;\omega)= \mathcal{X}(s,r;\omega)$. 
\end{itemize}
\end{definition}

\begin{remark}
{\rm Property $(i)$ means that for each $s$ and $x$, the solution $X^x_s(t)= \mathcal{X}(s,t;\cdot)(x)$ has continuous trajectories. This holds in the most of interesting cases by means of the so-called Da Prato-Kwapie\'n-Zabczyk factorization, see e.g. \cite{DaPrato-Zabczyk}. In particular, the solution has the required property  if  $(B_k)$ is a finite sequence of bounded operators on $H$.  Note, that property $(iii)$ of the stochastic flow follows from the existence and uniqueness of the solution. Finally, note that for all $0\le s\le t\le r\le h$ and $ x,y\in H$, the random variables $\mathcal{X}(s,t)(x)$  and $\mathcal{X}(r,h)(y)$ are independent. Stochastic flows with this property are known as Brownian flows.}
\end{remark}
The theory of stochastic flows for linear and nonlinear stochastic differential equations in finite dimensional spaces is well established, see for example \cite{Kunita}.  In particular, it is known  that the nonlinear SDE 
\begin{equation}\label{E12}
\d X= b(X)\d t + \sigma(X)\d W, \qquad X(s)=x, 
\end{equation}
defines a (nonlinear) stochastic flow if $b$ and $\sigma$ are $C^1$ with bounded derivatives. The proof of this result requires a Sobolev embedding theorem. The case of additive noise, where $\sigma$ does not depend on $X$, is easier. Namely, let $Y^x_s(t)$, $t\ge s$, be  the solution to  the ordinary differential equation with random coefficients 
\begin{equation}\label{E13}
\d Y(t)= b\left(Y(t)+\sigma (W(t)-W(s))\right)\d t, \qquad Y(s)=x. 
\end{equation}
Then the solution $X^x_s$ to \eqref{E12} is given by $X^x_s(t)= Y^x_s(t)+ \sigma (W(t)-W(s))$, and consequently \eqref{E12} defines a stochastic flow if \eqref{E13} defines a flow. 
\par
It is natural to ask a question about the existence of stochastic flow associated to a stochastic evolution equation 
\begin{equation}\label{E111}
\d X=(AX+F(X))\d t+\sum_k\sigma_k(X)\d W_k,\quad X(0)=x\,,
\end{equation}
where $H$ is an infinite-dimensional Hilbert space and $A$ is the  generator of a $C_0$-semigroup on $H$. If $\sigma$ does not depend on $X$, then one can apply the same argument as in in the finite dimensional case.  In this paper we focus on a linear equation \eqref{E11}, where $F=0$ and $\sigma_k(x)=B_kx$. Skorokhod showed in his famous example \cite{skorokhod} that the question of the existence of the flow for equation \eqref{E11} is much more intricate than in finite dimensions. In Section 4 we consider a more general version of the Skorokhod example to obtain new and interesting phenomena. 
\par
In general, it is not possible to prove the existence of the flow in infinite dimensions using the Sobolev embedding theorems that are not available in infinite dimensions. However, for many nonlinear equations of the form \eqref{E111}  the corresponding flows can be obtained as continuous transformations of the flows corresponding to ordinary or partial differential equations with random coefficients (see e.g. \cite{Brzezniak-Albeverio-Dalecki, Barbu-Rockner-1, Barbu-Rockner-2, Barbu-Rockner-Zhang, Brzezniak-Capinski, Brzezniak-vanNeerven, Brzezniak-Manna, ggn,Lisei}). In fact,  we will use a similar method in the proof of   our main result concerning linear equations.  A very general approach to the question of existence of stochastic flows for  non-linear stochastic partial differential equations can be found in  \cite{Flandoli-Lisei}. 
\par
It is tempting to obtain the existence of the flow $\mathcal X$ corresponding to linear equation \eqref{E11} by solving a linear equation 
\[\d \mathcal X= A\mathcal X\d t + \sum_{k} B_k \mathcal X\,\d W_k,\qquad t\ge s,\,\,\mathcal X(s)=I\,,\]
on the space of bounded operators. Unfortunately, the theory of stochastic integration on the space of all bounded linear operators on infinite dimensional Hilbert space   does not exist. It is possible however to integrate in some smaller spaces, such as the Schatten classes. This approach has been applied by Flandoli \cite{Flandoli}) in the case of Hilbert-Schmidt operators. In Section \ref{SSchatten} of this paper we use the theory of stochastic integration on $M$-type 2 Banach spaces to extend the Flandoli result to the scale of Schatten classes $\mathbb S^p$ for all $p\in[2,\infty)$. We will show that in the diagonal case solutions in the Schatten classes can exist $\P$-a.s. but not in the sense of expected value. In particular, in Proposition 4.3.3 we give conditions under which $\mathrm{Tr}(\mathcal X(t))<+\infty$,  $\P$-a.s. while $\E\,\mathrm{Tr}(\mathcal X(t))=+\infty$. It would be interesting to find more general conditions for such a behaviour of stochastic flows.
\par\medskip
The paper is organized as follows:  in the next section we prove a theorem dealing with equation \eqref{E11} with a sequence  of bounded operators $B_k$. We do not assume that the operators commute, therefore there are no explicit solutions. Then in Sections \ref{SDiagonal} we examine the case of commuting $B_k$. In this case  the solution is given in an explicit form, and this allows us to construct interesting examples that lead to new questions that remain open.  In Section  \ref{SSchatten} we study the existence of stochastic flows in Schatten classes. 


\section{Main results}
In this section we consider a linear stochastic equation on a separable Hilbert space $H$:
\begin{equation}\label{E21}
\d Y= B_0Y\d t + \sum_{k}B_kY\d W_k
\end{equation}
driven by a sequence of  independent Wiener processes $W_k$.  In theorems below we assume that $B_0, B_1, \ldots $ is a sequence of bounded linear operators on $H$ such that 
\begin{equation}\label{Assumption}
M:= \sum_{k} \|B_k\|<+\infty.
\end{equation}

We start with a result on the existence and regularity of the flow. 
\begin{theorem}\label{T21} 
Assume \eqref{Assumption}. Then equation \eqref{E21} 
defines a stochastic flow $\mathcal{Y}$  on $H$. Moreover, for any $T\in (0,+\infty)$, and $q\ge 1$, 
\begin{equation}\label{E22}
\sup_{0\le s\le t\le T} \mathbb{E}\left\| \mathcal{Y}(s,t)\right\|^{q} <+\infty.
\end{equation}
Moreover, for every $T\in (0,+\infty)$, and $L= 2,3,\ldots$, there exists a constant $C$ such that for all for $0\le s\le t\le u\le T$ we have  
\begin{equation}\label{E23}
\mathbb{E}\left\| \mathcal{Y}(s,t)- \mathcal{Y}(s, u)\right\|^{2L} \le C\left\vert t-u\right\vert ^{L-1}. 
\end{equation}
Consequently, by the Kolmogorov test, for any $s\ge 0$, $\mathbb{P}$-a.s.  $\mathcal{Y}(s,t)$ is a H\"older continuous $L(H,H)$-valued mapping of $t\ge s$, with exponent $\gamma <1/2$. Finally, for  any $T\in (0,+\infty)$, and $L= 2,3,\ldots$, there exists a constant $C$ such that for all for $0\le s\le t\le u\le T$ we have  
\begin{equation}\label{E24}
\mathbb{E}\left\| \mathcal{Y}(s,t)- \mathcal{Y}(u, v)\right\|^{2L} \le C\left( \left \vert u-s\right\vert + \left\vert t-v\right\vert\right)  ^{L-1}. 
\end{equation}
Consequently, by the Kolmogorov test, for any $s\ge 0$, $\mathbb{P}$-a.s.  $\mathcal{Y}(s,t)$ is a H\"older continuous $L(H,H)$-valued mapping of $(t,s)\in \Delta $, with exponent $\gamma <1/2$. 
 \end{theorem}
 \begin{theorem}\label{T21a}
Assume \eqref{Assumption}. Let  $\tilde B= \sum_{k\ge 1}B_k^2$, and  let $\mathcal{Y}$ be the flow defined by \eqref{E21}. Then for all $0\le s\le t$,  $\mathcal{Y}(s,t)$ is an invertible operator $\mathbb{P}$-a.s. and $\mathcal{Z}(s,t) = \left(\mathcal{Y}^{-1}(s,t)\right)^\star $ is the flow defined by  the equation 
\begin{equation}\label{E25}
\d Z= \left( -B_0^\star+ \tilde B^\star\right) Z\d t - \sum_{k} B_k^\star  Z \d W_k,\quad Z(s)=z\,.
\end{equation}
\end{theorem}

\begin{theorem}\label{T21b}
Assume \eqref{Assumption} Then for any generator $A$ of a $C_0$-semigroup on $H$,  equation \eqref{E11} defines stochastic flow on $H$. 
\end{theorem}
\begin{theorem}\label{T21c}
Assume that $B_0, B_1,\ldots, B_N$ is a finite sequence of commuting generators of $C_0$-groups on $H$. Then the equation 
$$
\d X= B_0X\d t + \sum_{k=1}^N B_kX\d_S W_k,\quad X(s)=x
$$
considered in the Stratonovich sense defines stochastic flow on $H$. Moreover, 
$$
\mathcal{X}(s,t)= \exp\left\{(t-s) B_0 + \sum_{k=1}^N \left(W_k(t)-W_k(s)\right)B_k\right\}. 
$$
\end{theorem}

\subsection{Proof of Theorem \ref{T21}} Write  $W_0(t)\equiv t$. Then \eqref{E21}  can be written as follows 
$$
\d Y= \sum_{k= 0}^{+\infty} B_kY\d W_k, \qquad Y(s)=x. 
$$
Let 
$$
\Theta:= \bigcup_{n=1}^{+\infty} \{0,1, \ldots\}^n. 
$$
Given $\alpha \in \{0,1,\ldots\}^n$ we set $\vert \alpha \vert =n$ and  $B^\alpha := B_{\alpha _1}B_{\alpha_2}\ldots B_{\alpha_n}$.  Let us fix an $s\ge 0$. We define by induction the iterated stochastic integrals. If $\vert \alpha \vert = 1$ and $\alpha = (k):= k$, $k=0,1,\ldots$, then $I_k(t)=W_k(t)-W_k(s)$. For $\alpha \in \{0,1,\ldots\}^n$ we set  
$$
I_\alpha(s,t):= \int_s^t I_{\widehat{\alpha }}(s,r)\d W_{\alpha_1}(r),
$$
where ${\widehat{\alpha }}:=(\alpha _2,\ldots, \alpha_n)$. Then the solution $Y^x_s$ to \eqref{E21} is given by 
$$
Y^x_s(t) = x+ \sum_{n=1}^{+\infty} \sum_{\alpha\colon \vert \alpha \vert =n}^{+\infty} I_\alpha(s,t)B^\alpha x. 
$$
Therefore the corresponding flow $\mathcal{Y}$ exists and 
$$
\mathcal{Y}(s,t;\omega)= I+ \sum_{n=1}^{+\infty} \sum_{\alpha\colon \vert \alpha \vert =n}^{+\infty} I_\alpha(s,t;\omega)B^\alpha  
$$
provided the series 
$$
\sum_{n=1}^{+\infty} \sum_{\alpha\colon \vert \alpha \vert =n}^{+\infty} I_\alpha(s,t;\omega)B^\alpha 
$$
converges $\mathbb{P}$-a.s in $L(H,H)$.  We have 
$$
\sum_{n=1}^{+\infty} \sum_{\alpha\colon \vert \alpha \vert =n}^{+\infty} \| I_\alpha(s,t;\omega)B^\alpha \|\le  \xi(s, t;\omega),
$$
where 
$$
\xi(s,t;\omega) := \sum_{n=1}^{+\infty}  \sum_{\vert \alpha \vert =n}^{+\infty} \left\vert I_\alpha (s, t;\omega)\right\vert \, \|B^\alpha\|.
$$
Therefore the stochastic flow exists if  $\xi(s,t)<+\infty $, $\mathbb{P}$-a.s. Moreover, $\left\|\mathcal{Y}(s,t;\omega)\right\| \le 1+ \xi(s,t;\omega)$. In fact we will show that for any $T\in (0,+\infty)$ and $L=1,2,3, \ldots $,  we have 
\begin{equation}\label{E26}
\sup_{0\le s \le t\le T} \mathbb{E}\, \xi^{2L}(s,t) < +\infty,
\end{equation}
which obviously guarantees  \eqref{E22}. To this end note that for $\alpha \colon \vert \alpha \vert=n \ge 1$ we have 
$$
\mathbb{E} \left\vert I_\alpha (s,t)\right\vert ^{2L}\le  C_L^{2Ln} \frac {\max\{ (1,  (t-s)^{2Ln}\}}{n!}(t-s)^{L-1}=:\eta_{n,L}(s,t), 
$$
where $C_L:=   \frac{2L}{2L-1}$.

Let $\alpha^1, \alpha^2,\ldots, \alpha^{2L} \in \Theta$.  By the Schwarz inequality we have 
$$
\mathbb{E} \prod_{i=1}^{2L} \left\vert  I_{\alpha^i} (s,t) \right\vert\le \prod_{i=1}^{2L} \left( \mathbb{E} \left\vert  I_{\alpha^i}(s,t)\right\vert^{2L}\right)^{1/(2L)}. 
$$
Therefore 
\begin{align*}
\mathbb{E}\, \xi^{2L}(s,t)& = \mathbb{E}\left(  \sum_{n=1}^{+\infty} \sum_{\alpha\colon \vert \alpha \vert =n}^{+\infty} \left\vert I_\alpha(s,t)\right\vert  \| B^\alpha\|\right)^{2L}\\
&\le  \sum_{n_1, \ldots n_{2L}=1}^{+\infty} \sum_{\alpha^i \colon \vert \alpha^i \vert =n_i}  \prod _{i=1}^{2L} \|B^{\alpha ^i} \|\,  \mathbb{E} \prod_{i=1}^{2L} \left\vert  I_{\alpha^i}(s,t)\right\vert\\
&\le \left( \sum_{n=1}^{+\infty} \sum_{\alpha \colon \vert \alpha \vert =n}^{+\infty} \|B^\alpha \|  \left( \mathbb{E} \left\vert  I_{\alpha}(s,t)\right\vert^{2L}\right)^{1/(2L)}\right)^{2L}
\\
&\le  \left( \sum_{n=1}^{+\infty} \left(\eta _{n,L}(s,t)\right)^{1/(2L)} \sum_{\alpha \colon \vert \alpha \vert =n}^{+\infty} \|B^\alpha \| \right)^{2L}
\end{align*}
Since 
\begin{align*}
\sum_{\alpha \colon \vert \alpha \vert =n}^{+\infty} \|B^\alpha \| &\le \sum_{\alpha \colon \vert \alpha \vert =n}^{+\infty} \|B_{\alpha_1}\|\|B_{\alpha _2}\|\ldots \|B_{\alpha _n}\| \\
&\le \left( \sum_{k=1}^{+\infty} \|B_k\|\right)^n=:M^n,
\end{align*}
we have 
\begin{align*}
\left( \mathbb{E}\, \xi^{2L}(s,t)\right)^{1/(2L)}  &\le \sum_{n=1}^{+\infty} M^n \left(\eta _{n,L}(s,t)\right)^{1/(2L)}\\
&\le  \sum_{n=1}^{+\infty} M^n C_L^{n} \left( \frac {\max\{ (1,  (t-s)^{2Ln}\}}{n!}\right)^{1/(2L)} (t-s)^{\frac 12 -\frac{1}{2L}} <+\infty,
\end{align*}
which gives  \eqref{E26}.  We will show  \eqref{E23} using calculations from the proof of \eqref{E26}. Namely, we have 
\begin{align*}
&\mathbb{E}\left \| \mathcal{Y}(s,t)- \mathcal{Y}(s,t+h)\right\|^{2L} = \mathbb{E}\left\| \sum_{n=1}^{+\infty} \sum_{\alpha\colon \vert \alpha \vert =n}^{+\infty} \left[ I_\alpha(s,t)- I_\alpha (s,t+h)\right] B^\alpha \right\|^{2L} \\
&\le  \left( \sum_{n=1}^{+\infty} \sum_{\alpha\colon \vert \alpha \vert =n}^{+\infty} \|B^\alpha \| \left( \mathbb{E} \left\vert  I_\alpha(s,t)- I_\alpha (s,t+h)\right\vert ^{2L}\right)^{1/(2L)}\right)^{2L}.
\end{align*}
Since 
$$
\mathbb{E}  \left\vert  I_\alpha(s,t)- I_\alpha (s,t+h)\right\vert^{2L} = \mathbb{E} \left\vert  I_\alpha(0,h)\right\vert^{2L},
$$
and \eqref{E26} holds,  for $0\le s\le t\le u\le T$ we have 
\begin{align*}
\mathbb{E}\left \| \mathcal{Y}(s,t)- \mathcal{Y}(s,u)\right\|^{2L} &= \mathbb{E}\left \| \mathcal{Y}(0,0)- \mathcal{Y}(0,u-t)\right\|^{2L}\\
&=  \mathbb{E}\, \xi^{2L}(0,u-t)\\
&\le  \left( \sum_{n=1}^{+\infty}  \left(MC_L\right)^{n}\left( \frac{\max\{1, T^{2Ln}\}}{n!}\right)^{1/(2L)} \right)^{2L} (u-t)^{L-1}. \qquad \square
\end{align*}

We are showing now \eqref{E24}. We can assume that $0\le s \le u$. Then there are three cases $(i)\colon s\le t\le u\le v$, $(ii)\colon s\le u\le t\le v$ and $(iii) \colon s\le u\le v\le t$. The first case  follows from \eqref{E23} by elementary calculations. For,  we have 
\begin{align*}
\mathbb{E}\, \|\mathcal{Y}(s,t)-\mathcal{Y}(u,v)\|^{2L} &\le 2^{2L}\left( \mathbb{E}\, \|\mathcal{Y}(s,t)\|^{2L} + \mathbb{E}\,  \|\mathcal{Y}(u,v)\|^{2L}\right)\\
&\le C\left( \vert t-s\vert ^{L-1}+ \vert v-u\vert ^{L-1}\right)\\
&\le C \left( \vert t-s\vert + \vert v-u\vert \right)^{L-1}= C \left( t-s +  v-u \right)^{L-1}\\
&\le C \left( v-t +  u-s \right)^{L-1}= C \left( \vert t-v\vert + \vert s-u \vert \right)^{L-1}. 
\end{align*}
The cases $(ii)$ and $(iii)$ follow also from \eqref{E22}, \eqref{E23},  and the flow property of $\mathcal Y$. Indeed consider $(ii)$. Then 
\begin{align*}
\|\mathcal{Y}(s,t)-\mathcal{Y}(u,v)\|&= \|\mathcal{Y}(u,t)\circ \mathcal{Y}(s,u)-\mathcal{Y}(t,v)\circ \mathcal{Y}(u,t)\|\\
&=\left\|\left( \mathcal{Y}(u,t) \circ  (\mathcal{Y}(s,u)- I\right) - \left(  \mathcal{Y}(t,v)-I\right) \circ \mathcal{Y}(u,t) \right\|\\
&\le \left\|\mathcal{Y}(u,t)\right\|  \left\|\mathcal{Y}(s,u)- \mathcal{Y}(s,s)\right\|  +  \left\|  \mathcal{Y}(t,v)-\mathcal{Y}(t,t)\right\| \left\|\mathcal{Y}(u,t)\right\| . 
\end{align*}
Therefore applying  the Schwarz inequality, and \eqref{E22} and \eqref{E23} for $4L$,  we can find a constant $C$ such that 
\begin{align*}
\mathbb{E}\, \|\mathcal{Y}(s,t)-\mathcal{Y}(u,v)\|^{2L} &\le C\left( |s-u|^{2L-1} + |t-v|^{2L-1}\right)^{1/2}\le  C'\left( |s-u| + |t-v|\right)^{L-1}. 
\end{align*}
Finally in  the last case (case $(iii)$), we have 
\begin{align*}
\|\mathcal{Y}(s,t)-\mathcal{Y}(u,v)\|&= \|\mathcal{Y}(v,t)\circ \mathcal{Y}(u,v)\circ \mathcal{Y}(s,u)-\mathcal{Y}(u,v)\|\\
&=\left\|  \mathcal{Y}(v,t)\circ \mathcal{Y}(u,v)\circ \left( \mathcal{Y}(s,u)- I\right)+ \left(  \mathcal{Y}(v,t)-I\right)\circ \mathcal{Y}(u,v) \right\|\\
&\le \left\|\mathcal{Y}(v,t)\right\|  \left\|\mathcal{Y}(u,v)\right\|  \left\|\mathcal{Y}(s,u)- \mathcal{Y}(s,s)\right\| +   \left\|  \mathcal{Y}(v,t)-\mathcal{Y}(v,v)\right\| \left\|\mathcal{Y}(u,v)\right\|. 
\end{align*}
Therefore applying  the Schwarz inequality, and \eqref{E22} and \eqref{E23} for $4L$ and $8L$, respectively,  we can find a constant $C$ such that 
\begin{align*}
\mathbb{E}\, \|\mathcal{Y}(s,t)-\mathcal{Y}(u,v)\|^{2L} &\le C\left( |s-u|^{2L-1} + |t-v|^{2L-1}\right)^{1/2}\le  C'\left( |s-u| + |t-v|\right)^{L-1}. 
\end{align*}
 $\square$.

\subsection{Proof of Theorem \ref{T21a}} From the first part we can easily deduce  the invertibility of $\mathcal{Y}(s,t)$. For, let us fix $0\le s\le t<+\infty$. Consider, the  partition $s=t^n_0< t^n_1<\ldots t_n^n=t$ such that $t^n_{k+1}-t^n_k = \frac{t-s}{n}$. From the H\"older continuity of $\mathcal{Y}$ it follows that with probability $1$, for any $\omega\in \Omega$, there is an $n(\omega)$ such that 
$$
\left\| \mathcal{Y}(t^n_k,t^{n}_{k+1};\omega)- \mathcal{Y}(t^n_{k}, t^n_k; \omega)\right\|\le \frac 12. 
$$
Hence, since $\mathcal{Y}(t^n_{k}, t^n_k; \omega)=I$, each $\mathcal{Y}(t^n_k,t^{n}_{k+1};\omega)$ is invertible. Since 
$$
\mathcal{Y}(s,t;\omega)=  \mathcal{Y}(t^n_{n-1},t^{n}_n;\omega)\circ \mathcal{Y}(t^n_{n-2},t^{n}_{n-1};\omega)\circ \ldots \circ \mathcal{Y}(t^n_{0},t^{n}_{1};\omega), 
$$
the invertibility of $\mathcal{Y}(s,t;\omega)$ follows.  

From the first part it also follows that \eqref{E25} defines stochastic flow $\mathcal{Z}$. We will first show  that  $\mathcal{Z}^\star(s,t) \mathcal{Y}(s,t)=I$. To this end note that for all $z,y\in H$, and $0\le s \le t$, we have 
\begin{align*}
\langle \mathcal{Z}^\star(s,t) \mathcal{Y}(s,t)y,z\rangle _H&= \langle  \mathcal{Y}(s,t)y, \mathcal{Z}(s,t)z\rangle_H = \langle Y^y_s(t), Z^z_s(t)\rangle_H.
\end{align*}
Next 
\begin{align*}
\d \langle Y^y_s(t), Z^z_s(t)\rangle_H&=  \langle B_0 Y^y_s(t), Z^z_s(t)\rangle _H\d t\\
&\quad + \langle  Y^y_s(t), (- B_0^\star +\tilde B^\star)Z^z_s(t)\rangle _H\d t   -\sum_k \langle B_kY^y_s(t), B_k^\star Z^z_s(t)\rangle  _H\d t \\
&\quad + \sum_{k} \langle B_kY^y_s(t), Z^z_s(t)\rangle _H\d W_k(t) -\sum_{k}\langle Y^y_s(t), B^\star_k Z^z_s(t)\rangle _H \d W_k(t)\\
&=0. 
\end{align*}
Hence for all $y,z\in H$ and $0\le s\le t$, 
$$
\langle \mathcal{Z}^\star(s,t) \mathcal{Y}(s,t)y,z\rangle _H= \langle y,z\rangle _H,
$$
and therefore $\mathcal{Z}^\star (s,t)\mathcal{Y} (s,t)=I$, $\mathbb{P}$-a.s.  Since $\mathcal{Y}(s,t)$ is invertible, $\mathcal{Z}^\star (s,t)= \left(\mathcal{Y}(s,t)\right)^{-1}$ as required. $\square$

\subsection{Proof of Theorem \ref{T21b}}  Recall that $(A,D(A))$ is the generator of a $C_0$-semigroup $\e^{tA}$, $t\ge 0$, on a Hilbert space $H$.  Let $A_\lambda= \lambda A(\lambda I -A)^{-1}$ be the Yosida approximation of $A$.  Consider the following approximation of \eqref{E11}, 
\begin{equation}\label{E27}
\d X=A_\lambda X\d t+\sum_{k}  B_kX\d W_k,\quad X(s)=x\in H. 
\end{equation}
Let $X^{x,\lambda}_s$ be the solution to \eqref{E27}. Recall that $X^x_s$ is the solution to \eqref{E11}.  It is easy to show that 
\begin{equation}\label{E28}
\lim_{\lambda\to +\infty}\mathbb{E}\left\vert X^{x,\lambda}_s(t)-X^x_s(t)\right\vert ^2_H=0, \qquad \forall\, t\ge s. 
\end{equation}
Define $ B_0 =\frac 12 \sum_{k=1}^{+\infty} B_k^2$.  Consider the following linear stochastic differential equation. 
\begin{equation}\label{E29}
\d Y= B_0Y \d t+\sum_{k} B_kY\d W_k,\qquad Y(s)=x\in H.
\end{equation}
From the first part of the theorem we know that \eqref{E27} and \eqref{E29} define stochastic flows, $\mathcal{X}_\lambda$ and $\mathcal{Y}$, respectively. Consider the following equation with random coefficients 
\begin{equation}\label{E210}
\d \mathcal{G}(t)= \mathcal{Y}(s,t)^{-1}\left(A_\lambda - B_0 \right)\mathcal{Y}(s,t)\mathcal{G}(t)\d t, \quad \mathcal{G}(s)=I.
\end{equation}
Taking into account H\"older continuity of $\mathcal{Y}$ one can see that the solution $\mathcal{G}_\lambda$ exists and 
$$
\mathcal{G}_\lambda(t,s)= \exp\left\{\int_s^t \mathcal{Y}(s,r)^{-1}\left(A_\lambda - B_0 \right)\mathcal{Y}(s,r)\d r\right\}.
$$
Moreover, see \cite{Tanabe}, there is a strongly continuous two parameter evolution system $\mathcal{G}(s,t)$ of bounded linear operators on $H$, such that for any $x\in H$, 
\begin{equation}\label{E211}
\lim_{\lambda \to +\infty} \mathcal{G}_\lambda(s,t)x= \mathcal{G}(s,t)x, \qquad \mathbb{P}-a.s. 
\end{equation}

We will show that  \eqref{E11} defines stochastic flow and $\mathcal {X}(s, t)= \mathcal{Y}(s,t) \mathcal{G}(s,t)$. Taking into account \eqref{E28} and \eqref{E211}  it is enough to show that  $X^{x,\lambda}_s(t)= \mathcal{Y}(s,t) \mathcal{G}_\lambda (s,t)x$, $t\ge s$, that is   that $\mathcal{Y}(s,t) \mathcal{G}_\lambda (s,t)x$, $t\ge s$, solves \eqref{E27}.  Clearly $\mathcal{Y}(s,s) \mathcal{G}_\lambda (s,s)x=x$. Next, for any $y\in H$ we have 
$$
\d \mathcal{Y}^\star(s,t)y= \mathcal{Y}^\star (s,t)B_0^\star y \d t + \sum_{k=1}^{+\infty} \mathcal{Y}^\star (s,t)B_k^\star y\d W_k(t).
$$
Therefore 
\begin{align*}
\d\left  \langle \mathcal{G}_\lambda (s,t)x, \mathcal{Y}^\star(s,t) y\right \rangle_H &= \left\langle \mathcal{Y}(s,t)^{-1}\left(A_\lambda - B_0 \right)\mathcal{Y}(s,t)\mathcal{G}_\lambda (s, t)x, \mathcal{Y}^\star(s,t) y\right \rangle_H \\
&\quad + \left\langle \mathcal{G}_\lambda (s,t)x, \mathcal{Y}^\star (s,t)B_0^\star y \right\rangle _H \d t\\
&\quad + \sum_{k=1}^{+\infty} \left\langle \mathcal{G}_\lambda (s,t)x,  \mathcal{Y}^\star (s,t)B_k^\star y\right\rangle _H\d W_k(t), 
\end{align*}
and consequently we have the desired conclusion 
$$
\d \mathcal{Y}(s,t) \mathcal{G}_\lambda (s,t)x=  A_\lambda  \mathcal{Y}(s,t) \mathcal{G}_\lambda (s,t)x\d t + \sum_{k=1}^{+\infty} B_k \mathcal{Y}(s,t) \mathcal{G}_\lambda(s,t)x\d W_k (t).  \qquad \square
$$

\subsection{Proof of Theorem \ref{T21c}} This part is well known. We present it only for complete presentation. 

\begin{remark}\label{R22}
{\rm The trick used in the proof of Theorem \ref{T21b} is well-known in finite dimensional case and in some infinite dimensional cases and is known as the Doss--Sussman transformation. 

There are important and interesting questions: $(i)$  whether the flow $\mathcal{X}$ defined by \eqref{E11} is H\"older continuous $L(H,H)$-valued mapping of $(s,t)\in \Delta$, $(ii)$ if the  flow is invertible. Clearly, see the proof of Theorem \ref{T21}(ii),  $(i)$ implies $(ii)$.  Unfortunately, if $A$ is unbounded it is probably impossible that the flow is continuous in the operator norm. Let us recall that the semigroup $\e^{tA}$ generated by $A$ is continuous in the operator norm if and only if $A$ is bounded. In the case when $\e^{tA}$, $t>0$,  are compact it is however possible that $\e^{tA}$ is continuous in the operator norm for $t>0$. Therefore, we can expect that in some cases the flow is continuous in the operator norm on the  open set $0<s<t<+\infty$, see Section \ref{SSchatten}.   Finally note that  if the flow $\mathcal{X}$ is invertible, than formally $\mathcal{Z}= \left(\mathcal{X}^{-1}\right)^*$ is the flow for the equation 
$$
\d Z= \left(-A^\star +\sum_{k} \left(B_k^2\right)^\star\right)Z\d t-\sum_{k} B_k^\star Z\d W_k. 
$$
Unfortunately, only in the case when $A$ generates a group,  $-A^\star$ generates a $C_0$-semigroup, and therefore the equation is often ill posed.}
\end{remark}
\begin{example}\label{ex_llg}
This example is an extension of a model that is important in the theory of random evolution of spins in ferromagnetic materials, see \cite{ggn}.

Let $\mathcal{O}\subset\R^3$ be a bounded domain. For a sequence $g_k\in C\la\mathcal{O},\R^3\ra$ we define operators 
$$
B_k\colon L^2\la\mathcal{O};\R^3\ra\to L^2\la\mathcal{O};\R^3\ra,\quad B_kx(\xi)=x(\xi)\times g_k\la\xi\ra,\quad \xi\in\mathcal{O}\,,
$$
where $\times$ stands for vector product in $\R^3$.  It is easy to see that  each $B_k$ is skew-symmetric, that is $B_k^\ast=-B_k$. 
Assume that $\sum_k \|g_k\|_{\infty}<+\infty$, hence $\sum_{k} \|B_k\|<+\infty$. Then, by Theorem \ref{T21},  the stochastic differential equation 
$$
\d Y=\sum_{k=1}^{+\infty}B_kY\d W_k,
$$
defines stochastic flow ${\mathcal Y}$ on $\bL$ and, by Theorem \ref{T21b}, $\mathcal{Z}(s,t) = \left(\mathcal{Y}^{-1}(s,t)\right)^\star $ is the flow defined by  the equation  
$$
\d Z= \tilde B Z\d t + \sum_{k} B_k  Z \d W_k. 
$$
where  $\tilde B= \sum_{k\ge 1}B_k^2$. 
\end{example}
\section{Nonlinear case}
The method based on the Doss--Susmann transformation can be generalized to some non-linear equations, see \cite{Flandoli-Lisei, Imkeller-Lederer1, Imkeller-Lederer2}. Assume that  $B_0,\ldots B_N \in C^1(\mathbb{R})$ is a finite sequence of  functions with bounded derivatives. Let $\mathcal{O}$ be an open subset of $\mathbb{R}^d$ and let $(A,D(A))$ be the generator of a $C_0$-semigroup on $H=L^2(\mathcal{O})$. Consider the following SPDE and SDE in the Stratonovich sense 
\begin{equation}\label{EN1}
\d X= \left[ AX+ B_0(X)\right]\d t +\sum_{k=1}^N B_k(X)\d_S W_k,\qquad X(s)=x\in L^2(\mathcal{O}), 
\end{equation}
\begin{equation}\label{EN2}
\d y=   B_0(y)\d t +\sum_{k=1}^N B_k(y)\d_S W_k(t),\qquad y(s)=\xi\in \mathbb{R}. 
\end{equation}
By the classical theory of SDEs, see e.g. \cite{Ikeda-Watanabe, Kunita},  \eqref{EN2} defines stochastic flow $\eta(s,t)$, $0\le s\le t$, of diffeomorphism  of $\mathbb{R}$.  Consider now the following stochastic evolution equation on the Hilbert space $L^2(\mathcal{O})$, 
\begin{equation}\label{EN3}
\d Y=  B_0(Y)\d t +\sum_{k=1}^N B_k(Y)\d_S W_k,\qquad Y(s)=x\in L^2(\mathcal{O}). 
\end{equation}
Then \eqref{EN3} defines stochastic flow $\mathcal{Y}$ on $L^2(\mathcal{O})$ and 
$$
\mathcal{Y}(s,t)(x)(\xi)= \eta(s,t)(x(\xi)), \qquad x\in L^2(\mathcal{O}), \quad \xi\in \mathcal{O}. 
$$
Let $G_s^x(t)$ be the solution to  the evolution equation with random coefficients 
$$
\frac{\d }{\d t}  G_s^x(t)= \mathcal{U}(s,t ,G_s^x(t)), \qquad G_s^x(s)=x, 
$$
where 
$$
\mathcal{U}(s,t,y)= \left[ D\mathcal{Y}(s,t)(y)\right]^{-1} A \mathcal{Y}(s,t)(y).  
$$
Above, $D\mathcal{Y}(s,t)(y)$ is the derivative with respect to initial condition $y$. Clearly we need to assume that $\mathcal{Y}(s,t)(G_s^x(t))$ is in the domain of $A$, and $D\mathcal{Y}(s,t)(G_s^x(t))$ is an invertible operator. 
Then by the  It\^o--Vencel  formula 
$$
X^x_s(t)= \mathcal{Y}(s,t)(G^x_s(t)). 
$$
In \cite{Flandoli-Lisei}, this method was applied to equations of the type 
$$
\d X= \mathcal{A}(X)\d t + \sum_{k=1}^N B_k X\d _S W_k, 
$$
where $\mathcal{A}$ is a monotone operator, and $(B_k)$ are first order differential linear operators.

\section{Diagonal and commutative case}\label{SDiagonal}
In this section we will first consider  an  extension of  the Skorokhod example \cite{skorokhod}.  Let $(e_k)$ be an orthonormal basis of an infinite-dimensional Hilbert space $H$ and let $(\sigma_k)$ and $(\alpha_k)$ be sequences of real numbers. We assume that $\sigma_k\ge 0$. For every $k\ge 1$ we define bounded linear operators
$B_k  = \sigma _k e_k\otimes e_k$,  and a possibly unbounded operator
$$
A= \sum_{j=1}^{+\infty} \alpha_j e_j\otimes e_j
$$
with the domain 
$$
D(A)=\left\{ x\in H\colon \sum_{j=1}^{+\infty} \alpha _j^2 \langle x,e_j\rangle ^2_H<+\infty\right\}. 
$$
Let us recall that in the Skorokhod example $\alpha_k=0$, $k=1,2,\ldots$, and $\sigma_k=\sigma$ does not depend on $k$. 
\begin{proposition}\label{P31}
Consider  \eqref{E11} with $A, B_1, \ldots $ as above. Then the following holds.\\
$(i)$ For each initial value $x\in H$, \eqref{E11} has  a  square integrable solution  $X^x_s$  in  $H$ if and only if
\begin{equation}\label{E31}
\sup_{k} \left(2\alpha _k + \sigma_k^2\right)<+\infty. 
\end{equation}
$(ii)$ Assume \eqref{E31}. Then \eqref{E11}  defines stochastic flow $\mathcal{X}$  if and only if  
\begin{equation}\label{E32}
\rho(s,t):= \sup_{k} \left\{\left[ \alpha _k -\frac{ \sigma_k^2}{2}\right] \sqrt{t-s} + \sigma_k \sqrt{2\log k}\right)<+\infty, \quad \forall\, 0\le s<t. 
\end{equation}
\end{proposition}
\begin{proof} Define 
$$
\zeta_k(s,t;\omega):= \exp\left\{ \sigma_k\left(W_k(t;\omega)-W_k(s;\omega)\right)+ \left(\alpha _k -\frac{\sigma_k^2}{2}\right)(t-s)\right\}. 
$$
Clearly random variable $\zeta_k(s,t)$, $k=1\ldots, $ are independent and 
\begin{equation}\label{E33}
\mathbb{E}\, \zeta_k(s,t)= \e^{\alpha_k(t-s)}\qquad\text{and}\qquad  \mathbb{E}\, \zeta^2_k(s,t)= \e^{(2\alpha_k+\sigma_k^2)(t-s)}. 
\end{equation}
Consider a  finite dimensional subspace $V=\text{linspan}\{e_{i_1},\ldots, e_{i_M}\}$ of $H$. Then for any $x\in V$, the solution exists and  
\begin{equation}\label{E34}
X^x_s(t)= \sum_{k=1}^{+\infty} \zeta_k(s,t)\langle x,e_k\rangle_H e_k. 
\end{equation}
Next, note that for any $x\in H$, if  $X^x_s$ is a solution, then it has to be of form \eqref{E34}. For, if $\Pi\colon H\mapsto \text{linspan}\{e_{i_1},\ldots, e_{i_M}\}$ is a projection then $\Pi X^x_s= X^{\Pi x}_s$. Therefore, the first claim follows from \eqref{E33}. Next, \eqref{E11} defines a  stochastic flow $\mathcal X$ if and only if 
$$
\mathbb{P}\left\{ \sup_{|x|_H\le 1} \left\vert X^x_s(t)\right\vert_H<+\infty\right\} = 1. 
$$
Clearly, 
\begin{align*}
&\mathbb{P}\left\{ \sup_{|x|_H\le 1} \left\vert X^x_s(t)\right\vert_H<+\infty\right\} = \mathbb{P}\left\{ \sup_k  \zeta(s,t)<+\infty\right\} 
\\&\qquad = 
\mathbb{P}\left\{ \sup_{k} \left[ \sigma_k\left(W_k(t)-W_k(s)\right)+ \left(\alpha _k -\frac{\sigma_k^2}{2}\right)(t-s) \right]<+\infty\right\}. 
\end{align*}
Therefore the desired conclusion follows from the fact that  if $(Z_k)$  is a sequence of independent $\mathcal{N}(0,1)$ random variables, then 
$$
\limsup_{k\to+\infty} \frac{Z_k}{\sqrt{2\log k}}=1,\qquad \mathbb{P}-a.s.
$$
\end{proof}

\begin{remark}
{\rm Assume \eqref{E31} and \eqref{E32}. Let $\mathcal{X}$ be the stochastic flow. Then 
$$
\mathbb{E}\, \text{Tr} \, \mathcal{X}(s,t)= \sum_{k=1}^{+\infty}\mathbb{E}\,  \zeta_k(s,t)= \sum_{k=1}^{+\infty} \e^{\alpha_k(t-s)}. 
$$
Note that \eqref{E31} and $\sum_{k=1}^{+\infty} \e^{\alpha_k(t-s)}<+\infty$ imply \eqref{E32}. }
\end{remark}
\subsection{Beyond integrability}
 As above, $(\sigma_k)$ is a sequence of nonnegative real numbers. Assume now that $\alpha_k=0$.  Thus 
$$
\zeta_k(s,t)= \exp\left\{ -\frac{\sigma^2_k}{2}(t-s)+ \sigma_k \left(W_k(t)-W_k(s)\right)\right\}. 
$$
Let us assume that 
\begin{equation}\label{E35}
\rho(s,t):= \sup_{k} \left\{-\frac{ \sigma_k^2}{2} \sqrt{t-s} + \sigma_k \sqrt{2\log k}\right)<+\infty, \quad \forall\, 0\le s<t. 
\end{equation}
We do not assume however \eqref{E31}. Then
$$
\mathcal{X}(s,t;\omega)(x)= \sum_{k=1}^{+\infty}\zeta_k(s,t;\omega)\langle x,e_k\rangle _He_k, \qquad 0\le s\le t, \ x\in H, 
$$
is well defined, and  since $\sup_{k} \zeta_k(s,t)<+\infty$, $\mathcal{X} \colon \Delta\times \Omega \mapsto L(H,H)$. Obviously $\mathcal{X}(s,t)$ is a symmetric positive definite operator with eigenvectors $(e_k)$ and eigenvalues $(\zeta_k(s,t))$. Note that it can happen that  $\mathbb{E}\left\vert \mathcal{X}(s,t)(x)\right\vert ^2_H=+\infty$. 
Moreover, if $\alpha_k=0$ then necessarily $\E\,\mathrm{Tr}\la \mathcal{X}(s,t)\ra=+\infty$.

\begin{proposition}\label{P33}
Assume  \eqref{E35}. Then the following holds:
\begin{itemize}
\item[(i)]  If $l\in[0,+\infty]$ is an accumulation point of the sequence $(\sigma_k)$,  then either $l=0$ or $l=+\infty$. 
\item[(ii)] If $\sigma_k\to 0$ then $\sup_k \sigma_k\sqrt{\log k} <+\infty$. In particular, $\mathbb{P}$-a.s. the sequence $\left(\zeta_k(s,t)\right)$ of eigenvalues of $\mathcal{X}(s,t)$ 
has accumulation points different from zero, hence $\mathcal{X}(s,t)$ is not compact $\P$-a.s. 
\item[(iii)] 
If $\sigma_k\to+\infty$ then $\frac{\sigma_k}{\sqrt{\log k}}\to+\infty$ and $\mathcal{X}(s,t)$ is Trace class $\mathbb{P}$-a.s.
 \end{itemize}
\end{proposition}
\begin{proof} Statement $(i)$ is obvious as $\sigma_k\ge 0$. The first part of $(ii)$ is obvious. Assume that $\sigma_k\to 0$ and that $\sup_{k}\sigma_k\sqrt{\log k}<+\infty$. We have to show that $\mathbb{P}$-a.s. the sequence 
$$
\exp\left\{ \sigma_k\left(W_k(t)-W_k(s)\right)\right\}, \qquad k=1,\ldots
$$
has accumulation points different from zero, or equivalently that 
$$
\liminf_{k\to +\infty} \sigma_k\left(W_k(t)-W_k(s)\right)>-\infty, \qquad \mathbb{P}-a.s. 
$$
Since 
$$
\liminf_{k\to +\infty} \frac{W_k(t)-W_k(s)}{\sqrt{2 \log k}}= -\sqrt{t-s}, \qquad \mathbb{P}-a.s.  
$$
we have 
$$
\liminf_{k\to +\infty} \sigma_k \left(W_k(t)-W_k(s)\right)= -\sqrt{t-s} \limsup_{k\to +\infty} \sigma_k \sqrt{2\log k}>-\infty, \quad \mathbb{P}-a.s. 
$$

We are proving now  the last statement od the proposition. Condition \eqref{E35} can be rewritten in the form 
$$
\sup_k\sigma_k^2\left(\frac{\sqrt{2\log k}}{\sigma_k}-\frac{\sqrt{t-s}}{2}\right)<+\infty,\qquad\forall\, 0\le s<t. 
$$
Since $\sigma_k\to +\infty$, we have 
$$
\limsup_{k\to +\infty} \frac{\sqrt{2\log k}}{\sigma_k}< \frac{\sqrt{t-s}}{2}, \qquad\forall\, 0\le s<t, 
$$
which leads to the desired conclusion that $\frac{\sigma_k}{\sqrt{\log k}}\to+\infty$. 

We will use the Kolmogorov three series theorem to show that 
$$
\text{Tr}\, \mathcal{X}(s,t)=\sum_{k=1}^{+\infty}\zeta_k(s,t)<+\infty,\qquad\mathbb{P}-a.s.
$$
Let us fix $s$ and $t$. Define $Y_k=\zeta_k(s,t)\chi_{[0,1]}\left(\zeta_k(s,t)\right)$. We need to show that 
\begin{align}
\sum_{k=1}^{+\infty}\mathbb{P}\left( \zeta_k(s,t)>1\right)&<+\infty,\label{E36}\\
\sum_{k=1}^{+\infty}\mathbb{E}\,  Y_k&<+\infty, \label{E37}\\
\sum_{k=1}^{+\infty}\mathrm{Var}\, Y_k&<+\infty\label{E38}
\end{align}
Denoting by $Z$ a normal  $\mathcal{N}(0,1)$ random variable and putting $b=\frac{\sqrt{t-s}}{2}$ we obtain 
\begin{align*}
\sum_{k=1}^{+\infty}\mathbb{P}\left( \zeta_k(s,t)>1\right) &=\sum_{k=1}^{+\infty}\mathbb{P}\left(Z>b\sigma_k\right) \\
&\le c+\dfrac{1}{\sqrt{2\pi}}\sum_{k=2}^{+\infty}\dfrac{1}{b\sigma_k}\e^{-b^2\sigma_k^2/2}=c+ \dfrac{1}{\sqrt{2\pi}}\sum_{k=2}^{+\infty}\dfrac{1}{b\sigma_k}\e^{-\delta_k\log k}\\
&\le c+C\sum_{k=2}^{+\infty}\dfrac{1}{k^{\delta_k}}
\end{align*}
with $\delta_k=\frac{b^2\sigma_k^2}{2\log k}$. Since $\delta_k\to+\infty$, \eqref{E36} follows. Let  $\left.\frac{\d\mathbb Q}{\d\mathbb{P}}\right|_{\mathfrak {F}_t}=\zeta_k(s,t)$,
where $\mathfrak{F}_t=\sigma\left(W_k(r), k=1,\ldots, r\le t\right)$. Then $W_k^{\mathbb Q}(r):=W_k(r)-\sigma_kr$, $r\le t$, are independent Wiener processes under $\mathbb Q$ and 
$$
\log\zeta_k(s,t)=\frac{1}{2}\sigma_k^2(t-s)+\sigma_k\left(W_k^{\mathbb Q}(t)-W_k^{\mathbb{Q}}(s)\right).
$$
Therefore, 
$$
\mathbb{E}\, Y_k =\mathbb{E}\, \zeta_k(s,t)\chi_{[0,1]}\left(\zeta_k(s,t)\right)= {\mathbb Q}\left(\zeta_k(s,t)\le 1 \right) 
=\mathbb Q\left( Z>\frac{1}{2}\sigma_k\sqrt{t-s}\right)
$$
and by the same arguments as above we find that \eqref{E37} holds as well. Finally, since 
$$
\mathbb{E}\, Y_k^2=\mathbb{E}\, \zeta^2_k(s,t)\chi_{[0,1]}\left(\zeta_k(s, t)\right) \le \mathbb{E}\,  Y_k,
$$
condition \eqref{E38} is satisfied and the proof is complete. 
\end{proof}
\subsection{The case of commuting operators}  Let us recall  that in the Skorokhod example  $A=0$ and $B_j=e_j\otimes e_j$ are commuting operators. We have 
\begin{align*}
X^x_s(t)&= \sum_{j=1}^{+\infty} \exp\left\{ W_j(t)-W_j(s)-\frac{t-s}{2}\right\} e_j\otimes e_j(x)\\
&= \lim_{N\to+ \infty} \sum_{j=1}^N \exp\left\{ W_j(t)-W_j(s)-\frac{t-s}{2}\right\} e_j\otimes e_j(x)\\
&= \lim_{N\to +\infty} \exp\left\{ \sum_{j=1}^N e_j\otimes e_j \left( W_j(t)-W_j(s)-\frac{t-s}{2}\right)\right\} (x).
\end{align*}
The  convergence is not uniform in $x$.  In the last line we calculate the exponent of a bounded operator 
$$
\sum_{j=1}^N e_j\otimes e_j \left( W_j(t)-W_j(s)-\frac{t-s}{2}\right).
$$
A proof the following simple generalization of Proposition \ref{P31} is left to the reader. Recall that a sequence of bounded operators $S_n$ converges to a bounded operator $S$ \emph{strongly} if $S_nx\to Sx$ for any $x\in H$. 
\begin{proposition}\label{P34}
Assume that $(B_k)$ is an  infinite sequence of bounded commuting operators on a Hilbert space  $H$. Then:
\begin{itemize}
\item[(i)] For any $x\in H$ and $s\ge 0$ there is a square integrable solution $X^x_s$ to \eqref{E11} if and only for all $0\le s\le t$, the sequence $\exp\left\{(t-s)\sum_{k=1}^n B_k^2\right \}$ converges strongly as $n\to +\infty$. Moreover, 
$$
X^x_s(t)= \lim_{n\to +\infty}\exp\left\{  \sum_{k=1}^n B_k \left(W_k(t)-W_k(s)\right)+ B_0(t-s)  - \frac{1}{2} \sum_{k=1}^n B_k^2  \left(t-s\right) \right\}x, 
$$
where the limit is in $L^2(\Omega, \mathfrak{F},\mathbb{P};H)$. 
\item[(ii)] \eqref{E11} generates a stochastic flow if and only if for all $0\le s\le t$, $\mathbb{P}$-a.s.
$$
\exp\left\{  \sum_{k=1}^n B_k \left(W_k(t)-W_k(s)\right) - \frac{1}{2} \sum_{k=1}^n B_k^2  \left(t-s\right) \right\}
$$
converges as $n\to+ \infty$ in the operator norm.  
\end{itemize}
\end{proposition}

\subsection{System of multiplication operators}
In this section we assume that  
$$
H=L^2(\mathbb{R}^d, \vartheta(\xi)\d \xi), 
$$
where the weight $\vartheta\colon \mathbb{R}^d\mapsto (0,+\infty)$ is a measurable function. Let $W$ be a Wiener process taking values in $H$. Then, see see e.g. \cite{DaPrato-Zabczyk},
\begin{equation}\label{E39}
W=\sum_{k} W_ke_k\,,
\end{equation}
where $(W_k)$ are independent real-valued Wiener processes,  and   $\{e_k\}$ is an orthonormal basis of the Reproducing Hilbert Kernel Space (RHKS for short)  of $W$.  

Consider the equation 
\begin{equation}\label{E310}
\d X(t)= X(t) \d W(t), \qquad X(s)=x, 
\end{equation}
Taking into account \eqref{E39} we can write \eqref{E310} in the  form 
\begin{equation}\label{E311}
\d X(t)= \sum_{k} B_k X\d W_k(t),\qquad X(s)=x,  
\end{equation}
where $B_k$ are multiplication operators; $B_kh= he_k$, $h\in H$.  Clearly a multiplication operator $h\mapsto he$ is bounded on $H$   if and only if $e\in L^\infty(\mathbb{R}^d)$. Note that bounded multiplication operators commute and are symmetric. 
\begin{proposition}\label{P35}
Assume that $e_k\in L^\infty(\mathbb{R}^d)$. Then:  
\begin{itemize}
\item[(i)] For any $x\in H$ and $s\ge 0$ there exists a square integrable solution $X^x_s$ to \eqref{E311} if and only if $\sum_{k=1}^{+\infty} e_k^2 \in L^\infty(\mathbb{R}^d)$. Moreover, 
$$
X^x_s(t)= L^2-\lim_{n\to +\infty} \exp\left\{  \sum_{k=1}^n \left(W_k(t)-W_k(s)\right)e_k   - \frac{t-s}{2} \sum_{k=1}^n e_k^2  \right\}x.  
$$
\item[(ii)] Assume that  $\sum_{k=1}^{+\infty} e_k^2 \in L^\infty(\mathbb{R}^d)$. Then \eqref{E311} defines a stochastic flow if and only if 
$$
\mathbb{P}\left\{  \operatorname{ess}\sup\limits_{\xi\in \mathbb{R}^d}\sum_k e_k(\xi)W_k(t)<+\infty\right\}=1, 
$$
equivalently iff the process $W(t)= \sum_k W_k(t)e_k $ lives in $L^\infty(\mathbb{R}^d)$; that is 
$$
\mathbb{P}\left( W(t)\in L^\infty(\mathbb{R}^d)\right)=1.
$$ 

\item[(iii)] If $\sum_{k} \log \sqrt{k} \vert e_k\vert \in L^\infty(\mathbb{R}^d)$, then \eqref{E311} defines stochastic flow on $H$. 
\end{itemize}
\end{proposition}
\begin{proof} The first part follows from Proposition \ref{P34}.  The second part is a reformulation of the second part of Proposition \ref{P34}. The las part follows from the law of iterated logarithm. 
\end{proof}
\begin{example}
Assume that $W=W(t,\xi)$ is a spatially homogeneous Wiener process on $\mathbb{R}^d$, see e.g. \cite{DaPrato-Zabczyk}.  Then $W(t)= \sum_{k} e_k W_{k}$, where  $e_k =\widehat{f_k\mu}$,  $\{f_k\}$ is a orthonormal basis of $L^2_{(s)}(\mathbb{R}^d,\mu)$ and $\mu $ is the spectral measure of $W$. The sum over  finite or infinite number of indices $k$.  We can assume that $e_k\in L^\infty(\mathbb{R}^d)$ by choosing suitable $f_k$. Then 
$$
\sum_{k} e_k^2 =  \sum_{k} \left\vert \widehat {f_k \mu}\right\vert ^2 =  \mu(\mathbb{R}^d). 
$$
Therefore, \eqref{E311} has a square  integrable  solution if and only if $\mu$ is finite, that is $W$ is a random field. The condition $\mathbb{P}\left( W(t)\in L^\infty(\mathbb{R}^d)\right)$ wich is if and only if condition for the existence of stochastic flow holds only in some degenerated cases. Namely assume that 
$$
W(t,\xi)= \sum_{k=1}^{+\infty} a_k \left(W_k(t)\cos\langle \xi, \eta_k\rangle + \tilde W_k(t)\sin \langle \xi,\eta_k\rangle\right),  
$$
where $\{\eta_k\}\subset\mathbb{R}^d$, $(a_k)\in l^2$,  and $W_k$ and $\tilde W_k$, $k\in \mathbb{N}$,  are independent real-valued Brownian motions. Then $W =W(t,\xi)$ is a spatially homogeneous Wiener process.  Taking into account that 
$$
\limsup_{k\to +\infty} \frac{W_k(t)}{\sqrt{2\log k}}= \sqrt{t}, 
$$
we see the equation defines stochastic flow on any weighted $L^2$-space if $\sum_{k} \vert a_k\vert \sqrt{\log k}<+\infty$. 
\end{example}
\begin{example}
Assume that $W$ is a Brownian sheet on $[0,L)^{d+1}$ where $L\in [0,+\infty]$.  To be more precisely $W(t, \xi_1, \xi_2,\ldots,\xi_d)$ is a Gaussian random field on $[0,L)^{d+1}$ with the  covariance 
$$
\mathbb{E}\, W(t,\xi_1,\ldots,\xi_d)W(s,\eta_1,\ldots,\eta_d)= t\wedge s \prod_{k=1}^d \xi_k\wedge \eta_k. 
$$
Then 
$$
f_k = \frac{\partial ^d}{\partial \xi_1,\ldots\partial \xi_d} e_k
$$
is an orthonormal basis of $L^2([0,L)^d)$, respectively.  Hence 
$$
\sum_{k} e_k^2(\xi) = \sum_k \langle \chi_{[0,\xi_1]\times \ldots \times [0,\xi_d]}, f_k \rangle ^2 =  \xi_1\xi_2\ldots \xi_d. 
$$
Therefore \eqref{E311} has a square  integrable  solution (in $L^2([0,T)^d$)  if and only if $L<\infty$.   Clearly, Brownian sheet has continuous trajectories, and therefore for arbitrary  $T<+\infty$ and $L<+\infty$ we have  $\mathbb{P}\left\{ \sup_{0\le t<T, \xi\in [0,L)^d} W(t, \xi)<+\infty\right\}=1$. Hence if \eqref{E311} is considered on a bounded domain the equation defines stochastic flow. Let us  observe that the stochastic flow can be also well defined, but not square integrable,  if $L=+\infty$.  Indeed, the stochastic flow exists, if and only if 
\begin{equation}\label{E412}
\mathbb{P}\left\{\sup\limits_{\xi\in [0,+\infty)^d}\left( W(t,\xi_1,\ldots,\xi_d)- \frac t 2 \xi_1\xi_2,\ldots \xi_d\right) <+\infty\right\}=1. 
\end{equation}
\end{example}

\section{Equations in Schatten classes}\label{SSchatten}
The problem of the existence of the stochastic flow in the finite dimensional case  $H=\mathbb{R}^d$ is simple. The  flow $\mathcal{X}(s,t)$ takes valued in the space of bounded linear operators $L(H,H)$ that can be identified with the space of $d\times d$ matrices $M(d\times d)$. We have the following SDE for $\mathcal{X}$ in the space $M(d\times d)$; 
\begin{equation}\label{E51}
\d \mathcal{X} = A\mathcal{X} \d t + (\d \mathcal{W})\mathcal{X}, \qquad \mathcal{X} (s,s)=I,
\end{equation}
where $\mathcal {W}:= \sum_{k=1}^d B_k W_k$  and $I$ is the identity matrix. By a standard fixed point argument \eqref{E51} has a unique global solution.
\par
 In infinite dimensional case, even if $(B_k)$ is a finite sequence of bounded linear operators, the situation is different. There is no proper theory of stochastic integration in the space $L(H,H)$ if $H$ is infinite-dimensional. One can overcome this difficulty by replacing $L(H,H)$ with a smaller space of operators, such as  the Hilbert--Schmidt or, more generally, the Schatten classes  of operators. 
\subsection{Main result}
Let us recall that  for every $p\in[1,+\infty)$ the Schatten class $\mathbb S^p$ of compact operators $K\colon H\to H$ is a Banach space endowed with the norm 
$$
\|K\|_p=\left(\sum_{k=1}^{+\infty} \lambda_k\left( K^\star K\right)^{p/2}\right)^{1/p}<+\infty\,,
$$
where $\lambda_k(K^\star K)$ stands for the $k$-th eigenvalue of $K^\star K$. For every $p\in[1,+\infty)$ the space $\mathbb S^p$ is a separable Banach space. 
\begin{lemma}\label{L41}
For every $p\ge 2$ the space $\mathbb S^p$ is an M-type 2 Banach space. 
\end{lemma}
\begin{proof}
By Propositions  5.4.2 in \cite{jan1} the space $\mathbb S^p$ is a UMD space for every $p\in(1,+\infty)$. By Proposition 7.1.11 in \cite{jan2},  $\mathbb S^p$ has type 2 for $p\in[2,+\infty)$, and by Proposition 4.3.13 in \cite{jan1},  M-type 2 property follows. 
\end{proof}

Lemma \ref{L41} ensures that if $p\ge 2$, then $\mathbb{S}^p$ is a right space for stochastic integration, for more details see e.g. \cite{Brzezniak}. We only recall here that if  $W_k$ are independent Wiener processes, ad  $\psi_k$ are $\mathbb{S}^p$-valued progressively measurable processes, such that 
$$
\mathbb{E} \int_0^T \|\psi_k(s)\|^2_p\d s <+\infty, \qquad k=1,2, \ldots,
$$
then for each $k$, the integral 
$$
\int_0^t \psi_k(s)\d W_k(s), 
$$
is well-defined, has continuous trajectories in $\mathbb{S}^p$ and there exists a universal constant $C$ such that 
$$
\mathbb{E}\left\|\int_0^T \psi_k (s)\d W_k(s)\right\| ^2_p \le C\, \mathbb{E} \int_0^T \|\psi_k(s)\|^2_p\d s <+\infty\,.
$$ 
Moreover, 
$$
\mathbb{E} \left\|\sum_{k=1}^N \int_0^t \psi_k(s) \d W_k(s)\right\| ^2_p \le C' \sum_{k=1}^N \mathbb{E} \int_0^T \|\psi_k(s)\|^2_p\d s. 
$$
 The definition of the Wiener process given below uses the fact that $L(H,H)=\la\mathbb S^1\ra^\star$. 
\begin{definition} 
Let $B_k\in L(H,H)$ for $k\ge 1$. We call $ \mathcal{W}=\sum_{k=1}^{+\infty} W_kB_k$  a \emph{cylindrical Wiener process} on $L(H,H)$  if for every $K\in\mathbb{S}^1$ the process 
$$
W^K(t)=\sum_{k=1}^{+\infty}\mathrm{Tr}\left( KB_k\right) W_k(t)\,,
$$
is a real-valued Wiener process, and there is a constant $C$ independent of $K$ such that 
\begin{equation}\label{Cyl}
\mathbb{E} \left|W^K(t)\right|^2=t \sum_{k=1}^{+\infty} \left(\mathrm{Tr}\left( KB_k\right)\right)^2\le Ct \|K\|_1^2<+\infty\,.
\end{equation}
\end{definition}
Note that condition \eqref{Cyl} holds if $\sum_k\left\|B_k\right\|^2_{L(H,H)}<\infty$. Indeed, we have (see Theorem 3.1 in \cite{simon}) 
$$
\left|\mathrm{Tr}\la KB_k\ra\right|\le \left\|KB_k\right\|_1\le \left\|B_k\right\|_{L(H,H)}\|K\|_1
$$
and the claim follows. 
\begin{lemma}\label{lem00}
Let $(A,D(A))$ be the generator of a $C_0$-semigroup $(\e^{tA})$  on $H$. For $T\in L(H,H)$ we define
\begin{equation}\label{E43}
\mathcal{S}(t)T=\e^{tA}\circ T,\quad  t\ge0\,,
\end{equation}
and denote $\mathcal S=\{\mathcal S(t);\, t\ge 0\}$. Then: \\
(a) $\mathcal{ S}=\left(\mathcal{ S}(t)\right)$ defines a semigroup (but in general not a $C_0$-semigroup) of bounded operators on $L(H,H)$.\\
(b) For every $p\in[1,+\infty)$ we have $\mathcal S(t)\mathbb S^p\subset\mathbb S^p$ and $\mathcal S$ defines a 
$C_0$-semigroup on $\mathbb S^p$. 
\end{lemma}
\begin{proof}
Since $\mathbb S^p$ is an ideal, we have  $\mathcal S(t)\mathbb S^p\subset\mathbb S^p$ and the operator $\mathcal S(t)\colon \mathbb S^p\to\mathbb S^p$ is bounded. To prove strong continuity of $\mathcal S$ on $\mathbb S^p$, let us recall that for every $p\in[1,+\infty)$ the space $\mathbb S^p$ is the closure of the space of finite rank operators in the $\mathbb S^p$-norm. Let $t>0$ and $T\in\mathbb S^p$ be fixed. There exists a sequence $\la T_n\ra$ of finite rank operators, such that $\left\|T_n-T\right\|_p\to 0$. Choose $n$ such that for $t\le 1$, 
$$
\left\|T-T_n\right\|_p+\left\|\e^{tA}\la T-T_n\ra\right\|_p<\varepsilon\,.
$$
Since 
\begin{align*}
\left\|\mathcal S(t)T-T\right\|_p&\le \left\|\mathcal S(t)\la T-T_n\ra\right\|_p+\left\|\mathcal S(t)T_n-T_n\right\|_p+\left\|T-T_n\right\|_p\\
&\le\varepsilon+\left\|\mathcal S(t)\la T-T_n\ra\right\|_p
\end{align*}
and 
\[\lim_{t\to 0}\left\|\mathcal S(t)T_n-T_n\right\|_p=0\,,\]
the strong continuity follows. 
\end{proof}
Let $A$ be the generator of a $C_0$-semigroup $(\e^{tA})$ on $H$ and let $\mathcal  A$ be the generator of the semigroup $\mathcal S$  defined in Lemma  \ref{lem00}.  Let $W$ be a cylindrical Wiener process on $L(H,H)$. Consider the following  stochastic equation on $\mathbb{S}^p$,  where $p\ge 2$, 
\begin{equation}\label{E42}
\d \mathcal{X} =\mathcal{A}\mathcal{X} \d t+(\d\mathcal{W})\mathcal{X}, \qquad \mathcal{X}(s)=I. 
\end{equation}
\begin{definition}
Let $p\ge 1$. We will say that a process $\mathcal{X}(s,\cdot;\cdot) \colon [s,+\infty)\times \Omega\mapsto L(H,H)$ with continuous paths in $L(H,H)$ is an $\mathbb{S}^p$-valued   solution to \eqref{E42} if $\mathcal{X}(s,\cdot;\cdot)\colon (s,+\infty)\times \mapsto \mathbb{S}^p$ is measurable and adapted, 
$$
\mathbb{E}\int_s^T \|\mathcal{X}(t)\|_p^2\d t <+\infty,\qquad \text{for any $T\in (s,+\infty)$}, 
$$
and 
$$
\mathcal{X}(t)=\mathcal{S}(t-s)I+\int_s^t\mathcal{S}(t-r)(\d\mathcal{W}(r))\mathcal{X}(r) \qquad \text{for  $t\ge s$, $\mathbb{P}$-a.s.}. 
$$
\end{definition}
In the equation above 
\[\int_s^t\mathcal{S}(t-r)(\d\mathcal{W}(r))\mathcal{X}(r)=\sum_k\mathcal S(t-r)B_k\mathcal X(r)\d W_k(r)\,.\]
We have the following result. 
\begin{proposition}
 If $\mathcal{X}$ is a solution to \eqref{E42} with  $\mathcal{A}$ as above, then  $\mathcal{X}$ is the flow corresponding to \eqref{E11}.  \end{proposition}
\begin{theorem}\label{T42} Assume that   there exist constants $\gamma <1/2$ and  $C>0$ such that 
\begin{equation}\label{E44}
\|S(t)\|_p\le\frac{C}{t^\gamma},\quad t\le 1. 
\end{equation}
Assume that $\mathcal{W}=\sum_kB_k W_k$  is a cylindrical    Wiener process on $L(H,H)$ such that $\sum_k \|B_k\|^2 <+\infty$. 
Then for any $s\ge 0$, \eqref{E42} has a unique solution in $\mathbb{S}^p$.  Moreover, $(s,+\infty)\ni t\mapsto \mathcal{X}(s,t)\in \mathbb{S}^p$ is continuous $\mathbb{P}$-a.s. 
\end{theorem}
\begin{proof}  Let us fix $0\le s<T<+\infty$. Let $\Psi$ be the space of all adapted measurable processes $\psi \colon (s,T] \times \Omega\mapsto \mathbb{S}^p$ such that 
$$
\mathbb{E}\int_s^T \|\psi(t)\|^2_p \d t <+\infty. 
$$
On $\Psi$ consider the family of equivalent norms 
$$
||| \psi ||| _{\beta} := \left(\mathbb{E}\int_s^T \e^{-\beta t} \|\psi(t)\|^2_p \d t\right)^{1/2}, \qquad \beta \ge 0.  
$$
Note that, as $(\e^{tA})$ is $C_0$ on $H$ there is a constant  $C_1<+\infty$ such that for $t\in (s,T]$,  and $\psi\in \Psi$, 
\begin{align*}
\sum_{k} \mathbb{E} \int_s^t \|\mathcal{S}(t-r)B_k \psi(r)\|_p^2 \d r &\le \sum_{k} \mathbb{E} \int_s^t \|\e^{(t-r)A}\|^2 \|B_k\|^2 \|\psi(r)\|_p^2 \d r \\
&\le C_1 \sum_{k} \|B_k\|^2 \,\mathbb{E} \int_s^t  \|\psi(r)\|_p^2 \d r <+\infty. 
\end{align*}
Therefore, by assumption \eqref{E44}, the mapping 
$$
\mathcal{I}\psi (t):= \mathcal{S}(t-s)I+\int_s^t\mathcal{S}(t-r)(\d\mathcal{W}(r))\psi(r), \qquad \psi\in \Psi,\quad t\in (s,T], 
$$
is well-defined from $\Psi$ into $\Psi$. For $\beta$ large enough $\mathcal I$ is contraction on $(\Psi,\vert \|\cdot\|\vert_\beta)$. For we have 
\begin{align*}
||| \mathcal {I}(\psi)-\mathcal{I}(\phi) |||^2 &\le C_2 \sum_{k}\|B_k\|^2 \mathbb{E} \int_s^T \e^{-\beta t} \int_s^t \|\psi(r)-\phi(r)\|^2_p\d r \d t\\
&\le C_3\mathbb{E} \int_s^T \e^{-\beta r}\|\psi(r)-\phi(r)\|^2_p\int_r^T \e^{-\beta(t-r)}\d t  \d r \\
&\le C_3 \frac{1}{\beta} ||| \psi-\phi |||^2_\beta.
\end{align*}
Thus by the Banach fixed point theorem there is an $\mathcal{X}(s,\cdot)\in \Psi $ such that $\mathcal {I} \left(\mathcal{X}(s,\cdot)\right)= \mathcal{X}(s,\cdot)$. What is left is to show that for any $\psi\in \Psi$, the stochastic integral 
$$
\int_s^t\mathcal{S}(t-r)(\d\mathcal{W}(r))\psi(r), \qquad t\ge s, 
$$
has continuous paths in $\mathbb{S}^p$. Since there exists an $\alpha >0$ such that 
$$
\int_s^T (t-s)^{-\alpha}\|\mathcal{S}(t-s) I\|^2_{p}\d t<+\infty, 
$$
the desired continuity follows from Burkholder inequality  by a standard modification of the Da Prato--Kwapien--Zabczyk factorization method. 
\end{proof}

\begin{example}
Assume that $A$ is diagonal $A= -\sum_{k} \alpha _k e_k\otimes e_k$, where $(e_k)$ is an orthonormal basis of $H$, and $\alpha_k\ge 0$ are real numbers. Then $A$ generates a $C_0$-semigroup $\e^{tA}$, $t\ge 0$,  on $H$.  Moreover, $\e^{tA}\in \mathbb{S}^p$ for $t>0$ if and only if 
$$
\|\e^{tA}\|_{p}= \left( \sum_{k} \e^{-p\lambda_k t}\right)^{1/p}<+\infty. 
$$
In particular, we can consider the heat semigroup generated by a Dirichlet Laplacian $\Delta$ in a bounded subdomain  of $\R^2$ with sufficiently smooth boundary. The eigenvalues of $\Delta$ have asymptotics $\lambda_{kn}\sim k^2+n^2$, hence 
$$
\sum_{k,n} \e^{-pt\la k^2+n^2\ra}\sim\frac{1}{2tp}\,,
$$
which yields 
$$
\|\e^{t\Delta} \|_{p}\sim\la\frac{1}{2tp}\ra^{1/p}\,.
$$
Therefore, condition \eqref{E44} is satisfied if and only if $p>2$. If $B_k\colon H\to H$ are bounded operators, such that $\sum_k\left\|B_k\right\|^2<+ \infty$,  then Theorem \ref{T42} ensures the existence of the stochastic flow  in $\mathbb{S}_p$ for any $p>2$, but the Hilbert--Schmidt theory developed in \cite{Flandoli} cannot be directly applied. Note however that  in \cite{Flandoli}, $B_k$ can be unbounded.
\end{example}

\end{document}